\documentclass[11pt,a4paper]{amsart}
\usepackage{amsmath,amssymb,tikz}

\newtheorem{theorem}{Theorem}%[section]

\begin{document}
\title{Critical digraphs with few vertices}
\author{Mat\v ej Stehl\'ik}
%\thanks{The second author was partially supported by ANR project Stint
%(ANR-13-BS02-0007), ANR project GATO (ANR-16-CE40-0009-01), and by LabEx
%PERSYVAL-Lab (ANR-11-LABX-0025).}
\address{Laboratoire G-SCOP, Universit\'e Grenoble Alpes, France}
\email{matej.stehlik@grenoble-inp.fr}
%\date{\today}

\begin{abstract}
  We show that every $k$-dichromatic vertex-critical digraph on at
  most $2k-2$ vertices has a disconnected complement.
  This answers a question of Bang-Jensen et al., and generalises a
  classical theorem of Gallai on undirected vertex-critical graphs.
\end{abstract}
\maketitle

\section{Introduction}

Neumann-Lara~\cite{NL82} defined the \emph{dichromatic number} $\chi(G)$ of a
digraph $G$ as the smallest number or colours needed to colour the vertices of
$G$ so that no colour class induces a directed cycle.
A digraph $G$ with $\chi(G)=k$ is said to be \emph{$k$-critical} (resp.\
\emph{$k$-vertex-critical}) if every proper subdigraph (resp.\ every proper induced
subdigraph) $G'\subset G$ satisfies $\chi(G')<k$.
The \emph{complement} of a digraph $G$ is the digraph $\overline G$
such that $V(\overline G)=V(G)$ and $(u,v) \in E(G)$ if and only if $(u,v) \notin E(G)$.

Bang-Jensen, Bellitto, Schweser and Stiebitz~\cite[Question~3]{BJBSS19+} asked
whether there exists a $k$-critical digraph on at most $2k-2$ vertices with a
connected complement.
The purpose of this note is to answer their question in the negative, by proving
the following theorem.

\begin{theorem}
  \label{thm:main}
  If $G$ is a $k$-vertex-critical digraph on at most $2k-2$ vertices, then the complement
  $\overline G$ is disconnected.
\end{theorem}

This generalises a classical theorem of Gallai~\cite{Gal63} on undirected vertex-critical
graphs; other proofs were given in~\cite{Mol99,Ste03}.
All the proofs of Gallai's theorem rely on the simple fact that a colouring of an undirected graph is
equivalent to a clique cover of its complement.
Since there is no such equivalence in digraphs, generalising the theorem to directed
graphs may initially seem out of reach.
Upon closer inspection, however, it becomes clear that the central role
is played by colour classes of size $1$ and $2$.
Exploiting the fact that two vertices induce an acyclic subdigraph of
$G$ if and only if they induce at least one arc in $\overline G$,
we are able to give a short proof of Theorem~\ref{thm:main} based on matching theory.

\section{The proof}

In what follows, a \emph{colouring} of a digraph $G=(V,E)$ is a partition
$\mathcal P$ of $V$ such that every part (called a \emph{colour class})
induces an acyclic subgraph.
A colouring $\mathcal P$ of $G$ is \emph{extreme} if it is optimal (i.e.,
$\mathcal P$ has $\chi(G)$ colour classes), and has the minimum number of
singleton colour classes among all optimal colourings.
Given a subset $X \subseteq V$, we denote the restriction of $\mathcal P$
to $G[W]$ by $\mathcal P[W]$.
Given two disjoint digraphs $G_1=(V_1,E_1)$ and $G_2=(V_2,E_2)$, the
\emph{disjoint union} of $G_1$ and $G_2$ is defined as
\[
  G_1 + G_2=\big(V_1 \cup V_2, E_1 \cup E_2\big).
\]

The key tool in our proof is the \emph{Gallai--Edmonds Theorem}~\cite{Edm65,Gal63,Gal64};
see also~\cite{LP86,Sch03}.
Given an undirected graph $H$, the \emph{Gallai--Edmonds decomposition}
of $H$ is defined as follows.
Denote by $D(H)$ the set of vertices missed by at least one maximum matching,
$A(H)$ the set of vertices in $V(H) \setminus D(H)$ with at least one neighbour in $D(H)$,
and $C(H)=V(H)\setminus (A(H)\cup D(H))$.

The following theorem can be found in~\cite[Theorem~3.2.1]{LP86}.
\begin{theorem}[The Gallai--Edmonds Theorem]
  \label{thm:g-e}
  If $H$ is a graph and $V(H)$ is partitioned into $A(H),C(H),D(H)$ as above, then
  \begin{enumerate}
    \item every maximum matching of $H$ contains a near-perfect matching
      of each component of $H[D(H)]$ and a perfect matching of $H[C(H)]$;
    \item every maximum matching of $H$ is of size $\tfrac12\big(|V(H)|+|A(H)|-c\big)$,
      where $c$ is the number of connected components of $H[D(H)]$.
  \end{enumerate}
\end{theorem}

\begin{proof}[Proof of Theorem~\ref{thm:main}]
  Let $G=(V,E)$ be a $k$-vertex-critical graph on at most $2k-2$ vertices.
  Fix an extreme $k$-colouring $\mathcal P$ of $G$;
  observe that there must be at least two singleton colour classes in $\mathcal P$.
  Let $U \subseteq V$ be the set of all vertices belonging to
  colour classes of size $1$ or $2$, and set $W=V\setminus U$.
  
  Define the simple undirected graph
  \[
    H=\big(U,\big\{uv:\{(u,v),(v,u)\} \not\subseteq E\big\}\big),
  \]  
  and consider the Gallai--Edmonds decomposition of $H$.
  Namely, let $D$ be the set of vertices of $H$ missed by at least one maximum matching,
  $A$ the set of vertices in $U \setminus D$ adjacent to $D$ in $H$,
  and $C=U\setminus(A \cup D)$.
  Let $c$ be the number of components of $H[D]$; we denote their vertex sets by
  $D_1, \ldots, D_c$.
  
  For every vertex $u \in D$, let $M_u$ be a maximum matching of $H$ which misses $u$.
  This corresponds to an extreme colouring $\mathcal M_u$ of $G[U]$: the unmatched vertices
  correspond to the singleton colour classes, and edges of $M_u$ correspond to colour classes
  of $\mathcal M_u$ of size $2$.
  We can now define the extreme colouring $\mathcal P_u$ of $G$, having $\{u\}$ as a colour class, by
  \[
    \mathcal P_u=\mathcal P[W] \cup \mathcal M_u.
  \]
  
  By Theorem~\ref{thm:g-e}, every maximum matching $M_u$ of $H$ contains a perfect
  matching of $H[C]$ and a near-perfect matching of each $H[D_i]$.
  This immediately implies that
  \begin{align}
    \chi(G[C])  &=\tfrac12|C|, \label{eq:chi(G[C])}\\
    \chi(G[D_i])&=\tfrac12(|D_i|+1), \text{ for every } 1\leq i \leq c. \label{eq:chi(G[D_i])}
  \end{align}
  Moreover, since $M_u$ misses at least two vertices of $U$, we must have $c \geq 2$.
  
  Now pick any $1 \leq i \leq c$, and let $u\in D_i$ and $v \in V\setminus (A\cup D_i)$.
  If $v \in C \cup (D \setminus D_i)$, then $u$ and $v$ are in different components
  of $H-A$, so $uv \notin E(H)$ and therefore $(u,v) \notin \overline E$ and $(v,u) \notin \overline E$.
  Now suppose that $v \in W$.
  Suppose for a contradiction that $(u,v) \in \overline E$ or $(v,u) \in \overline E$.
  Consider the extreme colouring $\mathcal P_u$ of $G$ having $\{u\}$ as a colour class,
  and consider the colouring $\mathcal P'=\mathcal P_u[V \setminus \{u,v\}]\cup\{u,v\}$.
  Since $(u,v) \notin E$ or $(v,u) \notin E$ (or both), the colour class $\{u,v\}$
  does not induce a directed cycle, so the colouring $\mathcal P'$ is proper, optimal, and has fewer
  singleton colour classes than $\mathcal P_u$, contradicting the hypothesis that
  $\mathcal P_u$ was extreme.
  Hence, $(u,v) \notin \overline E$ and $(v,u) \notin \overline E$ for all $u \in D_i$
  and all $v \in V\setminus (A\cup D_i)$.
  In other words, we have shown that
  \begin{equation}
    \label{eq:G-A}
    \overline G-A = \overline G[W\cup C] + \overline G[D_1] + \cdots + \overline G[D_c].
  \end{equation}
  This fact, in conjunction with~\eqref{eq:chi(G[C])} and~\eqref{eq:chi(G[D_i])},
  gives
  \begin{align}
  \begin{split}\label{eq:chi(G-A)}
    \chi(G-A)&=\chi(G[W\cup C])+\sum_{i=1}^c \chi(G[D_i])\\
             &=\chi(G[W])+\chi(G[C])+\sum_{i=1}^c \chi(G[D_i])\\
             &=\chi(G[W])+\tfrac12\big(|C|+|D|+c\big).
  \end{split}
  \end{align}
  
  On the other hand, by Theorem~\ref{thm:g-e}, every maximum matching
  of $H$ has size $\tfrac12\big(|U|+|A|-c\big)$, so
  \[
    \chi(G[U])=|U|-\tfrac12\big(|U|+|A|-c\big)=\tfrac12\big(|C|+|D|+c\big),
  \]
  and it follows that
  \begin{equation}
  \label{eq:chi(G)}
    \chi(G)=\chi(G[W])+\chi(G[U])=\chi(G[W])+\tfrac12\big(|C|+|D|+c\big).
  \end{equation}
  Comparing~\eqref{eq:chi(G-A)} and~\eqref{eq:chi(G)}, we see that $\chi(G-A)=\chi(G)$.
  Since $G$ is vertex-critical, we must have $A=\emptyset$.
  Substituting this into~\eqref{eq:G-A}, and recalling that $c\geq 2$, we conclude
  that $\overline G$ has at least two components.
  This completes the proof.
\end{proof}

\bibliographystyle{plain}

\end{document}